\documentclass[oneside,english]{amsart}
\usepackage[T1]{fontenc}
\usepackage[latin9]{inputenc}
\usepackage{amsthm}
\usepackage{amssymb}
\usepackage{esint}

\makeatletter
\numberwithin{equation}{section} %% Comment out for sequentially-numbered
\numberwithin{figure}{section} %% Comment out for sequentially-numbered
\theoremstyle{plain}
\theoremstyle{plain}
\newtheorem{thm}{Theorem}
  \theoremstyle{definition}
  \newtheorem{defn}[thm]{Definition}
  \theoremstyle{plain}
  \newtheorem{lem}[thm]{Lemma}
  \theoremstyle{plain}
  \newtheorem{prop}[thm]{Proposition}
  \theoremstyle{remark}
  \newtheorem*{rem*}{Remark}
  \theoremstyle{remark}
  \newtheorem{rem}[thm]{Remark}

\makeatother

\usepackage{babel}

\begin{document}
November 5, 2009, FM

\title[Symmetrization in Warped Products and Fiber Bundles]{Steiner and Schwarz Symmetrization in Warped Products and Fiber Bundles
with Density}

\maketitle
\begin{center}
Frank Morgan, Sean Howe, Nate Harman
\par\end{center}
\begin{abstract}
We provide very general symmetrization theorems in arbitrary dimension
and codimension, in products, warped products, and certain fiber bundles
such as lens spaces, including Steiner, Schwarz, and spherical symmetrization
and admitting density.
\end{abstract}

\section{\label{sec:Introduction}Introduction}

Symmetrization has always played a major role in geometry and analysis,
especially for the isoperimetric problem, but it is hard to provide
comprehensive statements and proofs. Steiner symmetrization in $\mathbb{R}^{N}$
replaces one-dimensional slices with centered intervals. Schwarz symmetrization
in $\mathbb{R}^{N}$ replaces $(N-1)$-dimensional slices with centered
balls. Generalized Schwarz symmetrization in $\mathbb{R}^{N}$ replaces
slices of some dimension $1\leq n\leq N-1$ with centered balls. These
results generalize readily to products $M\times\mathbb{R}^{n}$. Spherical
symmetrization in $\mathbb{R}^{N}$ replaces slices by spheres about
the origin with spherical caps.

Antonio Ros \cite[Sect. 3.2]{Ros-isoperimetric problem} gave a beautiful
proof of symmetrization in the context of manifolds with density.
Our first Proposition \ref{pro:sym for warped products} extends Ros
to warped products as asserted by Morgan \cite[Thm. 3.2]{Morgan-In Polytopes Small Balls about Some Vertex Minimize Perimeter }
and is general enough to include spherical symmetrization (Rmk. \ref{rem:SphericalSym(singular fibers)})
as well as Steiner and Schwarz symmetrization. Proposition \ref{pro:uniqueness for warped products}
treats the smooth case with an analysis of when equality holds after
Rosales \emph{et al.} \cite[Thm. 5.2]{Rosales et al - On the Isoperimetric Problem in Euclidean Space with Density}.
Propositions \ref{pro:sym for fiber bundles} and \ref{pro:smooth fiber bundles}
extend symmetrization to Riemannian fiber bundles with equidistant
fibers in which horizontal movement from fiber to fiber preserves
or scales volume. Some simple examples are lens spaces (fibered by
circles, Rmk. \ref{rem:LensSpaces}) as envisioned by Ros \cite[Thm. 2.11]{Ros-isoperimetric problem for lens spaces,Ros-isoperimetric problem},
similar Hopf circle fibrations of $\mathbb{S}^{2n+1}$ over $\mathbb{CP}{}^{n},$
the Hopf fibration of $\mathbb{S}^{7}$ by great $\mathbb{S}^{3}$s
and of $\mathbb{S}^{15}$ by great $\mathbb{S}^{7}$s. We were not,
however, able to complete the proof by symmetrization envisioned by
Vincent Bayle (private communication) to prove the still open conjecture
that in $\mathbb{R}^{N}$ with a smooth, radial, log-convex density,
balls about the origin are isoperimetric \cite[Conj. 3.12]{Rosales et al - On the Isoperimetric Problem in Euclidean Space with Density},
because horizontal movement from fiber to fiber does not preserve
or scale volume.

Standard references on symmetrization are provided by Burago and Zalgaller
\cite[Sect. 9.2]{Burago and Zalgaller- Geometric Inequalities} and
Chavel \cite[Sect. 6]{Chavel-Riemannian Geometry}. Gromov \cite[Sect. 9.4]{Gromov- isoperimetry of waists and concentration of maps}
after \cite[5.A]{Gromov Spectral} provides some sweeping remarks
and generalizations, including most of our results.

\subsection*{Acknowledgments}

This work began with the Williams College, National Science Foundation
SMALL undergraduate research Geometry Group at the University of Granada,
Spain, summer 2009. We would like to thank Vincent Bayle, Antonio
Cañete, Alexander Díaz, Rafael López, Manuel Ritoré, Antonio Ros,
César Rosales, and David Thompson for help and inspiration. We acknowledge
partial support by the National Science Foundation (grants to Morgan
and to the SMALL REU), the Spanish Ministerio de Educación y Ciencia
(grant to Ritoré, Rosales, \emph{et al.}), the University of Granada,
and Williams College.

\section{\label{sec:Symmetrization}Symmetrization}

A convenient general definition of perimeter in a metric space with
density is provided by Minkowski content or perimeter:
\begin{defn}
The \emph{Minkowski perimeter }of a region \emph{$R$} is the lower
right derivative \[
\liminf_{\Delta r\rightarrow0^{+}}\frac{\Delta V}{\Delta r}\]
for $r$ enlargements. In a Riemannian manifold with continuous metric
and density, the right limit $dV/dr$ exists and agrees with the usual
definition of perimeter as long as the boundary of $R$ is rectifiable
(see \cite[Thm. 3.2.39]{Federer-GMT}). We will use the following
routine lemma ($f$ in the source is the negative of our $f$.)\end{defn}
\begin{lem}
\label{lem:franks lemma}\emph{(cf.} \cite[Lem. 2.4]{Morgan- Isoperimetric estimates on products}\emph{)}
Let $f$, $h$ be real-valued continuous functions on $\left[a,b\right]$.
Suppose that the upper right derivative of $f$ and the right derivative
of h satisfy \[
\limsup_{\Delta x\rightarrow0^{+}}\frac{f(x+\Delta x)-f(x)}{\Delta x}\leq\lim_{\Delta x\rightarrow0^{+}}\frac{h(x+\Delta x)-h(x)}{\Delta x}.\]

\noindent Then $f(b)-f(a)\leq h(b)-h(a)$.\end{lem}
\begin{prop}
\label{pro:sym for warped products}\textbf{\emph{Symmetrization for
warped products.}} Let $B,F$ be smooth, complete Riemannian manifolds.
Consider a warped product $B\times_{g}F$ with continuous metric $ds^{2}=db^{2}+g(b)^{2}dt^{2}$
and continuous product density $\Phi(b)\cdot\Psi(t)$. Suppose that
for some $p\in F,$ geodesic balls about $p$ are isoperimetric. Let
$R$ be a region of finite (weighted) perimeter. Then the Schwarz
symmetrization $\mathrm{sym}(R)$, obtained by replacing the slice
in each fiber by a ball about $p$ of the same (weighted) volume,
has the same volume and no greater perimeter than $R$.\end{prop}
\begin{rem*}
Although typically $F$ is $\mathbb{R}^{n}$ or $\mathbb{S}^{n},$
there are many other possibilities such as the paraboloid $\{z=x^{2}+y^{2}\}$
\cite[Thms. 5, 8]{Benjamini}, \cite[Thm. 3.1(A)]{MHH}. Also balls
about the origin may be replaced by half-planes $\left\{ x_{n}\leq c\right\} $
(for $\mathbb{R}^{n-1}$$\times\mathbb{R}^{+}$ as well as $\mathbb{R}^{n}$)
when these have finite weighted volume. If $F$ is a space of revolution
about $p$, for all balls about $p$ to be isoperimetric or even stationary
(constant generalized mean curvature \cite[§18.3]{Morgan - GMT}),
the density $\Psi$ must be rotationally symmetric; for half-planes,
$\Psi$ must be a function of $x_{1},x_{2},...x_{n-1}$ times a function
of $x_{n}.$ \end{rem*}
\begin{proof}
The preservation of volume is just Fubini's theorem for warped products.

For small $r$, denote $r$-enlargements in $B\times_{g}F$ by a superscript
$r$ and $r$-enlargements in fibers by a subscript $r$. Consider
a slice $\left\{ b_{0}\right\} \times C$ $=R(b_{0})$ of $R$ and
a ball about the origin $\left\{ b_{0}\right\} \times D$ in the same
fiber of the same weighted volume. For general $b$, consider slices
$(\{b_{0}\}\times C)^{r}(b)$ of enlargements $(\left\{ b_{0}\right\} \times C)^{r}$
of $\left\{ b_{0}\right\} \times C$ and similarly slices $(\left\{ b_{0}\right\} \times D)^{r}(b)$
of enlargements $(\left\{ b_{0}\right\} \times D)^{r}$ of $\{b_{0}\}\times D$.
If $C$ were a single point, then $(\{b_{0}\}\times C)^{r}(b)=\left\{ b\right\} \times C_{r'}$
for some $r'$ independent of C, because the projection in F of a
shortest path $\gamma$ from $\left\{ b_{0}\right\} \times C$ to
a point in $\{b\}\times F$ is a shortest path $\gamma_{1}$ in $F$
and the length of $\gamma$ depends only on the length of $\gamma_{1}.$
Hence for any C, $(\{b_{0}\}\times C)^{r}(b)=\left\{ b\right\} \times C_{r'}$
and $(\left\{ b_{0}\right\} \times D)^{r}(b)=\left\{ b\right\} \times D_{r'}$
for the same $r'$. Because the fiber density $\Psi(t)$ is independent
of $b$, $\left\{ b\right\} \times C$ and $\left\{ b\right\} \times D$
have the same weighted volume. Since every $\left\{ b\right\} \times D_{s}$
is isoperimetric for given volume, the lower right derivative $dV/dr$
for the family $C_{r}$ is at least as great as the right derivative
$dV/dr$ for the family $D_{r}$, and hence the upper right derivative
$dr/dV$ for the family $C_{r}$ is no greater than the right derivative
$dr/dV$ for the family $D_{r}$. Consequently by Lemma \ref{lem:franks lemma},
if $C_{r_{1}}$ and $D_{r_{2}}$ have the same volume, $r_{1}\leq r_{2}$
; conversely, when $r_{1}=r_{2}=r'$, the volumes satisfy

\[
|\left\{ b\right\} \times D_{r'}|\leq|\left\{ b\right\} \times C_{r'}|.\]
Therefore \[
(\left\{ b_{0}\right\} \times D)^{r}(b)\subseteq\mathrm{sym}((\{b_{0}\}\times C)^{r}(b)).\]
Since this holds for all $b$, \[
(\left\{ b_{0}\right\} \times D)^{r}\subseteq\mathrm{sym}((\{b_{0}\}\times C)^{r}).\]
Since this holds for all $b_{0}$,

\[
(\mathrm{sym}(R))^{r}=\bigcup_{b_{0}}(\left\{ b_{0}\right\} \times D)^{r}\subseteq\mathrm{\bigcup_{b_{0}}sym}((\{b_{0}\}\times C)^{r})\subseteq\mathrm{sym}(R^{r}).\]
Consequently,\[
|(\mathrm{sym}(R))^{r}|\leq|R^{r}|\]
and $sym(R)$ has no more perimeter than $R$, as desired.\end{proof}
\begin{rem}
\textbf{\label{rem:SphericalSym(singular fibers)}Spherical symmetrization.}
In Proposition \ref{pro:sym for warped products}, at least for regions
of finite volume, one may allow singular fibers as long as the union
$S$ of such fibers has codimension 1 measure 0. An important example
is viewing $\mathbb{R}^{n+1}$ as the warped product $\{b\geq0\}\times_{b^{n}}\mathbb{S}^{n}$
with singular fiber $\{b=0\},$ yielding so-called \emph{spherical
symmetrization,} using spherical caps to replace slices by spheres
about the origin. To prove this generalization of Proposition \ref{pro:sym for warped products},
suppose that there were a counterexample. Then its restriction to
the complement of an appropriate small $r$-neighborhood of the singular
set $S$ would be a counterexample to the proof of Proposition \ref{pro:sym for fiber bundles}.
Note that the symmetrization of the restriction is just the restriction
of the symmetrization. Additional perimeter introduced by truncation
is negligible for most small $r$ by the finite volume hypothesis
\cite[§4.11]{Morgan - GMT}.
\end{rem}
Proposition \ref{pro:uniqueness for warped products} provides more
general symmetrization with uniqueness for regions (typically isoperimetric
regions) which satisfy certain smoothness hypotheses, as in the proof
by Rosales \emph{et al.} \cite[Thm. 5.2]{Rosales et al - On the Isoperimetric Problem in Euclidean Space with Density}
that in $\mathbb{R}^{n}$ with density $e^{r^{2}}$, balls about the
origin uniquely minimize perimeter for given volume. Chlebík \emph{et
al.} \cite{Chlebik et al-The perimeter inequality under Steiner symmetrization}
provide in Euclidean space an analysis of uniqueness without such
smoothness hypotheses. Proposition \ref{pro:uniqueness for warped products}
does not depend on Proposition \ref{pro:sym for warped products}.
\begin{prop}
\textbf{\emph{Smooth case with uniqueness}}.\label{pro:uniqueness for warped products}
Let $B,F$ be smooth Riemannian manifolds. Consider a smooth warped
product $B\times_{g}F$ with metric $ds^{2}=db^{2}+g(b)^{2}dt^{2}$
and product density $\Phi(b)\cdot\Psi(t)$. Suppose that for some
$p\in F,$ geodesic balls about $p$ are isoperimetric. Let $R$ be
a measurable set in $B\times_{g}F$. Suppose that its topological
boundary $\partial R$ meets almost every fiber smoothly (or not at
all). Let $R'$ denote its Schwarz symmetrization. Suppose that $\partial R'$
also meets almost every fiber smoothly and that its intersection with
other fibers contributes nothing to the area of $\partial R'$. Then
$R'$ has the same volume and no greater perimeter than $R$. If they
have the same perimeter and balls about the origin are uniquely isoperimetric
in $F$ (up to measure 0), then $R=R'$ up to a set of measure zero.\end{prop}
\begin{proof}
The preservation of volume is just Fubini's theorem for warped products.

Let $B_{0}$ be the set of points $b$ in $B$ for which $\partial R$
and $\partial R'$ are both smoothly transverse to the fiber over
$b$. By hypothesis, almost every fiber meets $\partial R$ and $\partial R'$
smoothly. By Sard's Theorem, almost every fiber meets $\partial R$
and $\partial R'$ transversely. Hence almost all points of $B$ lie
in $B_{0}.$ For now we consider $b\in B_{0}.$ Let $0\leq\theta<\pi/2$
denote the angle that $\partial R$ makes with the horizontal; at
each point of $\partial R$ let $v$ be the horizontal vector in that
direction with magnitude $\tan\theta.$ (If $\partial R$ is locally
the graph of a function $f:B\rightarrow F$, then $v=\nabla f$, and
the component of $v$ in any direction gives the rate of change of
$f$ in that direction.) The analogous $v'$ for $\partial R'$ has
constant magnitude in each fiber. Because the metric $g$ depends
only on $b$ and the density is a product density, horizontal movement
just scales volume. If we vary b, any additional change in volume
is due to $v.$ Since corresponding slices have the same volumes,
these changes must be the same in $R_{b}$ and $R'_{b}$:

\[
\mbox{(1)}\int_{\partial R_{b}}\pm v=\int_{\partial R'_{b}}v',\]
where the $\ensuremath{\pm}$ sign depends on the local orientation
of $\partial R_{b}$. The element of area $dA$ of $\partial R$ satisfies

\[
\mbox{(2) }dA=\int_{\partial R{}_{b}}\sqrt{1+v^{2}}\ db.\]
By Jensen's Theorem and the convexity of the function $h(x)=\sqrt{1+x^{2}},$

\[
\mbox{(3) }\hat{\int}_{\partial R_{b}}\sqrt{1+v^{2}}\geq\sqrt{1+\left(\hat{\int}_{\partial R_{b}}|v|\right)^{2}},\]
where the caret over the integral sign indicates normalization by
the measure of the domain of integration. Let $\rho$ denote the ratio
of the areas of $\partial R'_{b}$ and $\partial R_{b}$. By hypothesis
$\rho\leq1.$ By (1),

\[
\mbox{(4) }\sqrt{1+\left(\hat{\int}_{\partial R_{b}}|v|\right)^{2}}\geq\sqrt{1+\left(\rho\hat{\int}_{\partial R'_{b}}|v'|\right)^{2}}\geq\rho\sqrt{1+\left(\hat{\int}_{\partial R'_{b}}|v'|\right)^{2}},\]
because $h(x)=\sqrt{1+x^{2}}$ satisfies $h(\rho x)\geq\rho h(x)$
for any $0\leq\rho\leq1$, with equality only if $\rho=1$. Since
$|v'|$ is constant on $\partial R'_{b},$ 

\[
\mbox{(5) }\sqrt{1+\left(\hat{\int}_{\partial R'_{b}}|v'|\right)^{2}}=\hat{\int}_{\partial R'_{b}}\sqrt{1+v'^{2}}.\]
Finally,\[
\mbox{(6) }\int_{\partial R'_{b}}\sqrt{1+v'^{2}}\ db=dA'.\]
 Assembling inequalities (2)\textemdash{}(6) yields\[
dA\geq dA',\]
 with equality only if $\rho=1.$ Integration yields\[
|\partial R|\geq\int_{B_{0}}dA\geq\int_{B_{0}}dA'=|\partial R'|\]
because by hypothesis the slices over $B_{0}$ include almost all
of the area of $\partial R'.$ If equality holds, then for almost
all slices, $dA=dA',$ $\rho=1,$ and by the uniqueness hypothesis,
$R_{b}=R'_{b}$ (up to measure 0). Almost all other slices are empty.
Consequently $R=R'$ up to a set of measure zero. \end{proof}
\begin{rem*}
If we assume for example that $\partial R$ and $\partial R'$ are
smooth, then it follows that $R$=$R'$ .
\end{rem*}
The following proposition provides for certain fiber bundles associated
to Riemannian submersions a similar symmetrization in a related warped
product. A \emph{Riemannian submersion} $\pi:M\rightarrow B$ has
the property that $d\pi,$ restricted to the orthogonal complement
of its kernel, is an isometry. It follows that fibers are equidistant
and that locally parallel transport normal to one fiber yields a diffeomorphism
with any nearby fiber, which we assume preserves or scales volume.
\begin{prop}
\label{pro:sym for fiber bundles}\textbf{\emph{Symmetrization for
fiber bundles. }}Consider a smooth Riemannian submersion $M\rightarrow B$
with density $\Phi$ and a smooth warped product $B\times_{g}F$ with
product density $\Phi'$. Assume that geodesic balls about a fixed
point $p$ in $F$ are isoperimetric in every fiber of $B\times_{g}F$,
with no more perimeter than any competitor in the corresponding fiber
of $M$. Further assume that parallel transport normal to the fibers
from $M_{b_{1}}$ to $M_{b_{2}}$ scales volume on the fibers by $\left(g(b_{2})/g(b_{1})\right)^{n}$.
Suppose that $M$ is compact or more generally that: 
\begin{enumerate}
\item $B$ is compact or more generally has positive injectivity radius
and
\item for some $r_{0}>0$, for $r<r_{0}$ the $r$-tube about a fiber $M_{b}$
under parallel translation from that fiber has metric\[
ds^{2}=(1+o(1))(db^{2}+dt^{2})\]
Here db is in the direction of parallel translation; dt, which depends
on b, is the metric along fibers; and o(1) approaches 0 as r approaches
0, uniformly in b.
\end{enumerate}
Let $R$ be a region of finite perimeter. Consider the Schwarz symmetrization
$\mathrm{sym}(R)$ in the warped product $B\times_{g}F$, which replaces
the slice of $R$ in each fiber with a ball about $p$ of the same
volume in the corresponding fiber of $B\times_{g}F$. Then $\mathrm{sym}(R)$
has the same volume and no greater perimeter than $R$.\end{prop}
\begin{rem*}
For the simplest lens spaces (see Rmk. \ref{rem:LensSpaces}), $M=\mathbb{S}^{3}$
and $M_{b}=F=\mathbb{S}^{1}.$ In general, Gromov \cite[Sect. 9.4]{Gromov- isoperimetry of waists and concentration of maps}
suggests taking $F$ to be $\mathbb{R}$ or $\mathbb{R}^{+}$ with
an appropriate density for which balls about $p=0$ are isoperimetric
to reduce a product of well-understood factors to two dimensions.\end{rem*}
\begin{proof}
The preservation of volume is just Fubini's theorem for Riemannian
submersions.

As in the proof of Proposition \ref{pro:sym for warped products}
denote $r$-enlargements in $M$ by a superscript $r$ and $r$-enlargements
in fibers by a subscript $r$. Let $r$ be a small positive number
less than both $r_{0}$ and the injectivity radius of $B$. Consider
a slice $C=R(b_{0})$ of $R$ and a ball $D$ of the same volume about
the origin in the corresponding fiber of $B\times_{g}F$. For general
$b$, let $C'$ denote the image of $C$ in $M_{b}$ under normal
parallel transport and let $D'$ denote the copy of $D$ in $\left\{ b\right\} \times_{g}R^{n}$.
Such horizontal movement just scales volume, in $M$ by hypothesis
and in $B\times_{g}F$ because the density is by hypothesis a product
density. Therefore $|C'|=|D'|$. As in the proof of Proposition \ref{pro:sym for warped products},
$|D^{r}(b)|=|D'_{r'}|$, but due to the twisting in fiber bundles,
it is not necessarily true that $|C^{r}(b)|=|C'_{r'}|$. By the uniformity
hypothesis (2), the map by parallel transport based at $M_{b_{0}}$
from $B\times_{g}F$ to $M$ distorts the metric by $1+o(1)$, uniform
over $M$. Therefore \[
C'_{r'}\subseteq C^{r+o(r)}(b).\]
Since by hypothesis each $D'_{r'}$ has no more perimeter than any
competitor in the corresponding fiber of $M$, the lower right derivative
$dV/dr$ for the family $C'_{r'}$ is at least as great as the right
derivative $dV/dr$ for the family $D'_{r'}$, and hence the upper
right derivative $dr/dV$ for the family $C'_{r'}$ is no greater
than the right derivative $dr/dV$ for the family $D'_{r'}$. Consequently
by Lemma \ref{lem:franks lemma}, if $C_{r_{1}}$ and $D_{r_{2}}$
have the same volume, $r_{1}\leq r_{2}$ ; conversely, when $r_{1}=r_{2}=r'$,
the volumes satisfy

\[
|D'_{r'}|\leq|C'_{r'}|.\]
Hence \[
D^{r}(b)\subseteq\mathrm{sym}(C^{r+o(r)}(b)).\]
Since this holds for all $b$, \[
D^{r}\subseteq\mathrm{sym}(C^{r+o(r)}).\]
Since this holds for all $b_{0}$,\[
(\mathrm{sym}(R))^{r}=\bigcup_{b_{0}\in R}D^{r}\subseteq\bigcup_{b_{0}\in R}\mathrm{sym}(C^{r+o(r)})\subseteq\mathrm{sym}(R^{r+o(r)}).\]
Consequently,\[
|(\mathrm{sym}(R))^{r}|\leq|R^{r+o(r)}|,\]
and $\mbox{sym}(R)$ has no more perimeter than $R$, as desired.\end{proof}
\begin{rem}
\label{rem:LensSpaces}As in Remark \ref{rem:SphericalSym(singular fibers)},
at least for regions of finite volume, one may allow singular fibers
where the projection fails to be a submersion, as long as the union
$S$ of such fibers has codimension 1 measure 0. An example is the
fibration of $\mathbb{S}^{3}$ in $\mathbb{C}^{2}$ by orbits of the
action mapping $(z_{1},z_{2})$ to $(e^{ik\theta}z_{1},e^{il\theta}z_{2})$.
If $k=l=1$, this is just the smooth Hopf fibration, but for general
integers $k,l$ the orbits through (1, 0) or (0, 1) are singular.
These circle fibrations lift to the lens spaces and generalize to
all odd dimensions. 

Similarly the uniformity hypothesis (1) on $B$ is not necessary.
The second uniformity hypothesis (2) probably is not necessary (see
Prop. \ref{pro:smooth fiber bundles}), but our method of proof seems
to need it. Of course if $R$ is compact then hypotheses 1 and 2 are
not necessary. 

In the unwarped case, including the lens spaces, the $1+o(1)$ factor
in the proof is unnecessary, because distance from a point in a fixed
fiber in the lens space is no greater than its value in the product
(see \cite[Prop. 8.6]{Report}).

Another interesting example is the cone over an $n$-dimensional Riemannian
manifold $M_{0},$ which can be viewed as a warped product $\{b\geq0\}\times_{b^{n}}M_{0}.$
If regions in $M_{0}$ have no less perimeter than balls of the same
volume in $S^{n},$ the cone can be compared to $\{b\geq0\}\times_{b^{n}}\mathbb{R}^{n}=\mathbb{R}^{n+1}$
\cite[§3]{Morgan-In Polytopes Small Balls about Some Vertex Minimize Perimeter }.
\end{rem}
The following proposition relaxes the uniformity hypotheses of Proposition
\ref{pro:sym for fiber bundles} for smooth regions and provides a
framework for the analysis of when equality holds.
\begin{prop}
\label{pro:smooth fiber bundles}\textbf{\emph{Smooth case for fiber
bundles}}. Consider a smooth Riemannian submersion $M\rightarrow B$
with density and a smooth warped product $B\times_{g}F$ with product
density. Assume that geodesic balls about a fixed point $p$ in $F$
are isoperimetric in every fiber of $B\times_{g}F$, with no more
perimeter than any competitor in the corresponding fiber of $M$.
Further assume that parallel transport normal to the fibers from $M_{b_{1}}$
to $M_{b_{2}}$ scales volume on the fibers by $\left(g(b_{2})/g(b_{1})\right)^{n}$.
Let $R$ be a measurable set in $M$. Suppose that its topological
boundary $\partial R$ meets almost every fiber smoothly (or not at
all). Let $R'$ denote its Schwarz symmetrization in the warped product
$B\times_{g}F,$ which replaces the slice of $R$ in each fiber with
a ball about $p$ of the same volume. Suppose that $\partial R'$
also meets almost every fiber smoothly and that its intersection with
other fibers contributes nothing to the area of $\partial R'$. Then
$R'$ has the same volume and no greater perimeter than $R$.\end{prop}
\begin{rem*}
One may allow singular fibers, as long as the union of such fibers
has codimension 1 measure 0. In cases where there is always twisting,
as in the lens space examples of Remark \ref{rem:LensSpaces}, if
$R'$ has the same perimeter as $R$, then $R$ must be a union of
fibers.\end{rem*}
\begin{proof}
{[}short version{]} The proof is almost identical to the proof of
Proposition \ref{pro:uniqueness for warped products}. We cannot hypothesize
that M has a \emph{product} density, so we have added the hypothesis
that parallel transport scales volume. Equation (2) becomes a (favorable)
inequality due to the twisting in the fiber bundle. Everything else
remains the same.

{[}long version{]} The proof is almost identical to the proof of Proposition
\ref{pro:uniqueness for warped products}. The preservation of volume
is just Fubini's theorem for Riemannian submersions.

Let $B_{0}$ be the set of points $b$ in $B$ for which $\partial R$
and $\partial R'$ are both smoothly transverse to the fiber over
$b$. By hypothesis, almost every fiber meets $\partial R$ and $\partial R'$
smoothly. By Sard's Theorem, almost every fiber meets $\partial R$
and $\partial R'$ transversely. Hence almost all points of $B$ lie
in $B_{0}.$ For now we consider $b\in B_{0}.$ Let $0\leq\theta<\pi/2$
denote the angle that $\partial R$ makes with the horizontal; let
$v$ be the horizontal vector in that direction with magnitude $\tan\theta.$
The analogous $v'$ for $\partial R'$ has constant magnitude in each
fiber. For every $b$ the slices $R_{b}$ and $R'_{b}$ have the same
volume. Horizontal movement just scales volume, in $M$ by hypothesis
and in $B\times_{g}F$ because by hypothesis the density is a product
density. If we vary b, any additional change in volume is due to $v.$
Since corresponding slices have the same volumes, these changes must
be the same in $M$ and in $B\times_{g}F$:

\[
\mbox{(1)}\int_{\partial R_{b}}\pm v=\int_{\partial R'_{b}}v',\]
where the $\ensuremath{\pm}$ sign depends on the local orientation
of $\partial R_{b}$ and the integrals include the densities. The
element of area $dA$ of $\partial R$ satisfies

\[
\mbox{(2) }dA\geq\int_{\partial R{}_{b}}\sqrt{1+v^{2}}\ db,\]
the inequality due to any twisting in the fiber bundle; here we are
including the density in the integral over $\partial R{}_{b}$, not
in $db$. By Jensen's Theorem and the convexity of the function $h(x)=\sqrt{1+x^{2}},$

\[
\mbox{(3) }\hat{\int}_{\partial R_{b}}\sqrt{1+v^{2}}\geq\sqrt{1+\left(\hat{\int}_{\partial R_{b}}|v|\right)^{2}},\]
where the caret over the integral sign indicates normalization by
the measure of the domain of integration. Let $\rho$ denote the ratio
of the areas of $\partial R'_{b}$ and $\partial R_{b}$. By hypothesis
$\rho\leq1.$ By (1),

\[
\mbox{(4) }\sqrt{1+\left(\hat{\int}_{\partial R_{b}}|v|\right)^{2}}\geq\sqrt{1+\left(\rho\hat{\int}_{\partial R'_{b}}|v'|\right)^{2}}\geq\rho\sqrt{1+\left(\hat{\int}_{\partial R'_{b}}|v'|\right)^{2}},\]
because $h(x)=\sqrt{1+x^{2}}$ satisfies $h(\rho x)\geq\rho h(x)$
for any $0\leq\rho\leq1$, with equality only if $\rho=1$. Since
$|v'|$ is constant on $\partial R'_{b},$ 

\[
\mbox{(5) }\sqrt{1+\left(\hat{\int}_{\partial R'_{b}}|v'|\right)^{2}}=\hat{\int}_{\partial R'_{b}}\sqrt{1+v'^{2}}.\]
Finally,\[
\mbox{(6) }\int_{\partial R'_{b}}\sqrt{1+v'^{2}}\ db=dA';\]
 again here for consistency the density is included in the integral
over $\partial R'_{b}$, not in $db$. Assembling inequalities (2)\textemdash{}(6)
yields\[
dA\geq dA',\]
 with equality only if $\rho=1.$ Integration yields\[
|\partial R|\geq\int_{B_{0}}dA\geq\int_{B_{0}}dA'=|\partial R'|\]
because by hypothesis the slices over $B_{0}$ include almost all
of the area of $\partial R'.$\end{proof}

\author{\noindent Frank Morgan, Department of Mathematics and Statistics, }

\author{\noindent Williams College, Williamstown, MA 01267}

\emph{E-mail address}: Frank.Morgan@williams.edu
\\

\author{\noindent Sean Howe, Department of Mathematics,}

\author{\noindent University of Arizona, Tucson, AZ 85721}

\emph{E-mail address}: seanpkh@gmail.com
\\

\author{\noindent Nate Harman, Department of Mathematics and Statistics,}

\author{\noindent University of Massachusetts, Amherst, MA 01003}

\emph{E-mail address}: nateharman1234@yahoo.com
\end{document}